\documentclass[oneside,english]{article}
\usepackage[T1]{fontenc}
\usepackage{amsthm}
\usepackage{amsmath}
\usepackage{amssymb}
\usepackage[all]{xy}
\usepackage{amsmath}
\usepackage{comment}
\usepackage{color}
\usepackage[nocompress]{cite}
\usepackage{tikz-cd}
\usetikzlibrary{decorations.markings}
\tikzset{double line with arrow/.style args={#1,#2}{decorate,decoration={markings,%
			mark=at position 0 with {\coordinate (ta-base-1) at (0,1pt);
				\coordinate (ta-base-2) at (0,-1pt);},
			mark=at position 1 with {\draw[#1] (ta-base-1) -- (0,1pt);
				\draw[#2] (ta-base-2) -- (0,-1pt);
}}}}

\makeatletter
\newcommand{\nc}{\newcommand}
\newtheorem{thm}{Theorem}
\theoremstyle{plain}
\nc{\bthm}{\begin{thm}} \nc{\ethm}{\end{thm}}
\newtheorem{prop}[thm]{Proposition}
\nc{\bprp}{\begin{prop}} \nc{\eprp}{\end{prop}}
\newtheorem{prob}[thm]{Problem}
\nc{\bprb}{\begin{prob}} \nc{\eprb}{\end{prob}}
\newtheorem{fact}[thm]{Fact}
\nc{\bfct}{\begin{fact}} \nc{\efct}{\end{fact}}
\newtheorem{lem}[thm]{Lemma}
\nc{\blem}{\begin{lem}} \nc{\elem}{\end{lem}}
\newtheorem{claim}[thm]{Claim}
\nc{\bclm}{\begin{claim}} \nc{\eclm}{\end{claim}}
\newtheorem{cor}[thm]{Corollary}
\nc{\bcor}{\begin{cor}} \nc{\ecor}{\end{cor}}
\newtheorem{conj}[thm]{Conjecture}
\nc{\bcnj}{\begin{conj}} \nc{\ecnj}{\end{conj}}
\theoremstyle{definition}
\newtheorem{defn}[thm]{Definition}
\nc{\bdfn}{\begin{defn}} \nc{\edfn}{\end{defn}}
\newtheorem{observation}[thm]{Observation}
\nc{\bobs}{\begin{observation}} \nc{\eobs}{\end{observation}}
\theoremstyle{remark}
\newtheorem{rem}[thm]{Remark}
\nc{\brem}{\begin{rem}} \nc{\erem}{\end{rem}}
\newtheorem{cnv}[thm]{Convention}
\nc{\bcnv}{\begin{cnv}} \nc{\ecnv}{\end{cnv}}
\newtheorem{exam}[thm]{Example}
\nc{\bexm}{\begin{exam}} \nc{\eexm}{\end{exam}}

\newtheorem*{answer*}{Answer}

\nc{\bpf}{\begin{proof}} \nc{\epf}{\end{proof}}
\nc{\be}{\begin{enumerate}}
	\nc{\ee}{\end{enumerate}}
\nc{\bi}{\begin{itemize}}
	\nc{\itm}{\item}
	\nc{\ei}{\end{itemize}}
\nc{\invlim}{\lim_{\leftarrow}}
\nc{\dirlim}{\lim_{\rightarrow}}
\nc{\mm}{\mathbf{m}}
\nc{\nn}{\mathbf{n}}
\nc{\FF}{\mathcal{F}}
\nc{\CC}{\mathcal{C}}
\nc{\Span}{\operatorname{span}}
\nc{\Img}{\operatorname{Im}}
\nc{\rank}{\operatorname{rank}}
\nc{\proj}{\operatorname{proj}}
\nc{\F}{\mathbb{F}}
\nc{\Z}{\mathbb{Z}}
\nc{\Q}{\mathbb{Q}}
\nc{\cd}{\operatorname{cd}}
\nc{\Ind}{\operatorname{Ind}}

\title{Pro-$p$ Poincar\'{e}-Duality groups of infinite rank}
\author{Tamar Bar-On}
\date{}

\begin{document}
	
	\maketitle
	\begin{abstract}
		We introduce the class of \textit{Generalized Poincar\'{e}-Duality groups}: i.e, pro-$p$ groups of infinite rank which satisfy a Poincar\'{e}-duality. We prove some basic properties of Generalized Poincar\'{e}-Duality groups, and show that under some conditions, a group that fits into an exact sequence of GPD groups must be a Demushkin group.
	\end{abstract}
	\section*{Introduction}
	Fix a prime $p$. Pro-$p$ Poincar\'{e} groups of dimension $n$, also known as pro-$p$ Poincar\'{e} duality groups of dimension $n$, are the infinite pro-$p$ groups $G$ satisfying the following conditions:\begin{enumerate}

		\item $\dim H^n(G)=1$.
		\item $\dim H^i(G)<\infty$ for all $0\leq i\leq n$.
		\item The cup product induces a nondegenerate pairing $H^i(G)\times H^{n-i}(G)\to H^n(G)$ for all $0\leq i\leq n$.
		
	\end{enumerate}
	We will refer to the third condition as a "Poincar\'{e} Duality" .
	
	Here and below $H^i(G)$ always denotes the $i$'th cohomology group of $G$ with coefficients in $\F_p$, considered as a module with the trivial action. 
	Notice that since for every pro-$p$ group $G$, $\dim H^1(G)$ equals the rank of $G$, i.e, the minimal cardinality of a set of generators converging to 1, then  pro-$p$ Poincar\'{e} duality groups (which we refer to as pro-$p$ PD groups) must be finitely generated. Pro-$p$ PD groups were studied in detail by Serre (\cite{serre1979galois}) and are in fact a special case of the more general definition of  \textit{profinite} Poincar\'{e} duality groups, which are in turn a generalization of the abstract Poincar\'{e} duality groups. For more information on profinite PD groups one shall look at \cite[Chapter 3]{neukirch2013cohomology}. 
	
	Pro-$p$ PD groups play a crucial role in several aspects of pro-$p$ group theory. One of the most famous and important examples of pro-$p$ PD groups are the uniform analytic pro-$p$ groups (See a theorem of Lazard, also can be found in \cite[Chapter 11]{wilson1998profinite}), which implies that every $p$-adic analytic group is virtually a pro-$p$ PD group. In addition, the pro-$p$ PD groups of dimension 2 are precisely the Demushkin groups, which cover all maximal pro-$p$ Galois groups of local fields that contain a primitive root of unity of order $p$ (see results of Serre and Demushkin in \cite{serre1962structure, demushkin1961group, demushkin19632}), and by the Elementary type conjecture by Ido Efrat, serve as part of the building blocks of all finitely generated maximal pro-$p$ Galois groups of fields which contain a primitive root of unity of order $p$. (\cite{efrat1995orderings}). 
	
	In 1986, Labute extended the theory of Demushkin groups to that of "Demushkin groups of infinite countable rank" which were defined as the pro-$p$ groups of infinite countable rank, i.e, $\dim(H^1(G))=\aleph_0$, which satisfy $\dim H^2(G)=1$ and for which the cup product bilinear form $H^1(G)\times H^1(G)\to H^2(G)\cong \F_p$ is nondegenerate. In \cite{bar2024demushkin} the theory was extended to Demushkin groups of arbitrary rank, which are the pro-$p$ groups of arbitrary rank which satisfy that $H^2(G)\cong \F_p$ and that the cup product bilinear form $H^1(G)\times H^1(G)\to H^2(G)\cong \F_p$ is nondegenerate.. In particular, it has been proven that for every uncountable cardinal $\mu$, there exit $2^{\mu}$ pairwise nonisomorphic Demushkin groups of rank $\mu$. 
	
	The object of this paper, that will hopefully lead to further research, is to extend the theory of Demushkin groups of arbitrary rank, and start to develop a theory for pro-$p$ Poincar\'{e}-Duality groups of arbitrary rank, which we refer to as \textit{Generalized Poincar\'{e} Duality} groups (GPD groups).
	\begin{defn}
		Let $G$ be an infinite pro-$p$ group. We say that $G$ is a generalized Poincar\'{e} Duality group of dimension $n$ if $H^n(G,\mathbb{F}_p)\cong \mathbb{F}_p$ and for every $0\leq i\leq n$, the cup product yields a nondegenerate pairing: $H^i(G,\mathbb{F}_p)\times H^{n-i}(G,\F_p)\to H^n(G,\mathbb{F}_p)\cong \mathbb{F}_p$. 
	\end{defn}
	Notice that although for pro-$p$ PD groups, the nondegeneracy implies that the natural maps $H^i(G)\to H^{n-i}(G)^*$ induced by the cup product are isomorphisms, for GPD groups we can only get injections, which makes the theory more complicated. In fact, requiring the maps $H^i(G)\to H^{n-i}(G)^*$ to be isomorphisms will imply that $G$ is finitely generated (see, for example, the proof of \cite[Proposition 3.7.6]{neukirch2013cohomology}), hence injectivity is the strongest property we can require.  
	
	The paper is organized as follows: in the first section we give some general results, and in particular prove that the cohomological dimension of a GPD group of dimension $n$ is $n$, and compute its dualizing module at $p$. In Section 2 we prove that under some conditions, a group that fits into an exact sequence of GPD groups must be a Demushkin group.

	\section{Basic properties}
	In this section we compute the cohomological dimension of GPD groups of dimension $n$, and the possible options for the dualizing module at $p$- whose existence is guaranteed by the finiteness of the cohomological dimension. For the rest of this section, $G$ denotes a generalized Poincar\'{e}-Duality group of dimension $n$.
	
	We first need a few lemmas. Let ${_p\operatorname{Mod}(G)}$ denotes the class of $G$- modules $A$ which are annihilated by $p$, i.e, for which $pA=0$.
	\begin{lem}\label{from F_p to every module}
		For every finite $G$- module $A\in {_p\operatorname{Mod}(G)}$, the natural maps $$H^i(G,A)\to H^{n-i}(G,A^*)^*$$ induced by the cup product $$H^i(G,A)\times H^{n-i}(G,A^*)\to H^n(G,A\otimes A^*)\to H^n(G,\F_p)\cong \F_p$$ are injective, for $i=0,1$.
	\end{lem}
	\begin{rem}
		Observe that for a $G$-module $A\in {_p\operatorname{Mod}(G)}$, $A^*=\operatorname{Hom}(A,\F_p)$, and $H^i(G,A)\in {_p\operatorname{Mod}(G)}$ as well. Hence the above maps are defined.
		For the rest of the paper the maps $$H^i(G,A)\to H^{n-i}(G,A^*)^*$$ will always refer to the maps induced by the cup product.
	\end{rem}
	\begin{proof}[Proof of Lemma \ref{from F_p to every module}]
		First observe that every $A\in {_p\operatorname{Mod}(G)}$ is a vector space over $\F_p$ and hence of $p$-power cardinality.
		
		Now let $A\in {_p\operatorname{Mod}(G)}$ be finite.	We need to show that for every $0\leq i\leq 1$, the map $H^i(G,A)\to (H^{n-i}(G,A^*))^*$ induced by the cup product, is injective. First we prove the claim for $i=1$. We prove it by induction on the size of $A$. For $|A|=p^1$, this is the assumption. Assume that $|A|>p$.  Since the only simple module over a pro-$p$ group is $\F_p$, $A$ admits an exact sequence $0\to A_0\to A\to \F_p\to 0$. Hence we get a commutative diagram 
		$$
		\xymatrix@R=14pt{ H^0(G,\F_p) \ar[d] \ar[r] & H^1(G,A_0) \ar[d] \ar[r] &  H^1(G,A) \ar[d] \ar[r] & H^1(G,\F_p) \ar[d]  \\
			H^n(G,\F_p)^* \ar[r] & H^{n-1}(G,(A_0)^*)^*  \ar[r] &  H^{n-1}(G,A^*)^* \ar[r] & H^{n-1}(G,\F_p^*)^*  }
		$$
		
		By induction assumption, and by definition, the maps $$H^1(G,A_0)\to (H^{n-1}(G,A_0^*))^*$$ and $$H^1(G,\F_p)\to (H^{n-1}(G,\F_p^*))^*$$ are injective. Since the map $H^0(G,\F_p)\to (H^{n}(G,\F_p^*))^*$ is in fact isomorphism, by the cardinalities of the groups, then diagram chasing implies  that the map $H^1(G,A)\to (H^{n-1}(G,A^*))^*$ is injective. Now that we have the injectivity in $H^1$ for every module $A\in {_p\operatorname{Mod}(G)}$, we will prove the injectivity for every module $A\in {_p\operatorname{Mod}(G)}$ in $H^0$. First we want to show that the functor $H^0(G,)$ in the category of finite $G$- modules in ${_p\operatorname{Mod}(G)}$ is coeffacable. That has been done in \cite[Proof of Proposition 3.7.6, Page 218]{neukirch2013cohomology} and the proof holds for every infinite pro-$p$ group.  Hence for every $A$ we can choose some finite $G$-module $B\in {_p\operatorname{Mod}(G)}$ with a projection $B\to A$ such that the induced map $H^0(B)\to H^0(A)$ is zero. Hence looking at the commutative diagram   
		$$
		\xymatrix@R=14pt{ H^0(G,B) \ar[d] \ar[r] & H^0(G,A) \ar[d] \ar[r] &  H^1(G,A^0)\ar[d]  \\
			H^n(G,B^*)^* \ar[r] & H^{n}(G,A^*)^*  \ar[r] &  H^{n-1}(G,{A^0}^*)^*}
		$$
		when $A^0=\ker(B\to A)$. By the injectivity of $H^1(G,A^0)\to  H^{n-1}(G,{A^0}^*)^*$, which holds for every finite $A^0\in {_p\operatorname{Mod}(G)}$, we get the injectivity of $H^0(G,A)\to  H^{n}(G,{A}^*)^*$. 
	\end{proof}
	\begin{lem}\label{Isomorphism for 0}
		For every finite $G$- module $A\in {_p\operatorname{Mod}(G)}$, the natural map $H^0(G,A)\to H^n(G,A^*)^*$ induced by the cup product is in fact an isomorphism.
	\end{lem}
	\begin{proof}
		We already know that the map $H^0(A)\to H^n(A^*)^*$ is injective. We are left to show that it is surjective. 
		
		We prove it by induction on the size of $A$. If $|A|=p^1$ it follows from equality of finite cardinalities. Now assume that $|A|>p$. 
		
		Look at the exact sequence $0\to A_0\to A\to \F_p$. It yields the following diagram:
		$$
		\xymatrix@R=14pt{ H^0(G,A_0) \ar[d] \ar[r] & H^0(G,A) \ar[d] \ar[r] &  H^0(G,\F_p)\ar[d] \ar[r] & H^1(G,A_0) \ar[d]\\
			H^n(G,{A_0}^*)^* \ar[r] & H^{n}(G,A^*)^*  \ar[r] &  H^{n}(G,{\F_p}^*)^*\ar[r] & H^{n-1}(G,{A_0}^*)^* }
		$$
		By the injectivity of the last vertical map, and the surjectivity of the first and third ones, we conclude the required surjectivity.
	\end{proof}
	Following the last Lemma, we can compute the cohomological dimension of $G$, using the same proof as in (\cite[Proposition 3.7.6]{neukirch2013cohomology})
	\begin{cor}\label{cohomological dimension}
		Let $G$ be a GPD group of dimension $n$. Then $cd(G)=n$
	\end{cor}
	\begin{proof}
		We already know that $cd(G)\geq n$ since $H^n(G,\F_p)=\F_p$. Now we want to prove that $H^{n+1}(G,\F_p)=0$. Let $x\in H^{n+1}(G,\F_p)$. Since every continuous map from a profinite group to a finite set projects through some finite quotient, there is an open subgroup $U$ such that $\operatorname{res}^G_U(x)=0$, which means that the map $H^{n+1}(G,\F_p)\to H^{n+1}(G,\operatorname{Ind}^G_U(\F_p))$ sends $x$ to 0. By Lemma \ref{Isomorphism for 0} and the facts that $H^0(G,-)$ is a left exact functor while taking the dual space $(-)^*$ is an exact functor, we conclude that the functor $H^n(G,-)$ is right exact on ${_p\operatorname{Mod}(G)}$. Thus, taking the exact sequence $0\to \F_p\to  \operatorname{Ind}^G_U(\F_p)\to A\to 0$ we get the exact sequence $H^n(G,\operatorname{Ind}^G_U(\F_p))\to H^n(G,A)\to H^{n+1}(G,\F_p)\to H^{n+1}(G,\operatorname{Ind}^G_U(\F_p))$. Since the map $H^n(G,\operatorname{Ind}^G_U(\F_p))\to H^n(G,A)$ is onto we conclude that the map $ H^{n+1}(G,\F_p)\to H^{n+1}(G,\operatorname{Ind}^G_U(\F_p))$ is injective. Thus, $x=0$.
	\end{proof}
	Recall that every pro-$p$ group of finite cohomological group $n$ admits a dualizing module at $p$ (sometimes refered to simply as a dualizing module), $I$, which is defined as $I={\dirlim}_{i\in \mathbb{N}}{\dirlim}_{U\leq o G}H^n(U,\Z/p^i\mathbb{Z})^*$, with the dual maps of the corestrictions, and satisfies $H^n(G,A)^*\cong \operatorname{Hom}_G(A,I)$ for every $p$-torsion $G$- module $A$ (see \cite[Theorem 3.4.4]{neukirch2013cohomology}). Observe that from the duality property of the dualizing module $I$ of a pro-$p$ group of cohomological dimension $n$ one can conclude that $_pI\ne 0$, for example by putting $A=\F_p$. For pro-$p$ PD groups, the dualizing module at $p$ is known to be isomorphic to $\Q_p/\Z_p$. In fact, pro-$p$ PD groups are precisely The pro-$p$ \textit{Duality groups} (see \cite[Theorem 3.4.6]{neukirch2013cohomology} for the definition) for which the dualizing module at $p$ is $\Q_p/\Z_p$. Observe that the action on $\Q_p/\Z_p$ might be nontrivial. In his paper on Demushkin groups of countable rank (\cite{labute1966demuvskin}) Labute has proved that the dualizing module for a Demushkin group of countable rank must be any of the following options: $\Q_p/\Z_p, \Z/p^s$ for every natural number $s$. Moreover, every such option occurs as the dualizing module at $p$ for some Demushkin group of rank $\aleph_0$. This result was extended to Demushkin groups of arbitrary rank in (\cite{bar2024demushkin}). We generalize this result for GPD groups:  
	\begin{prop}\label{dualizing module}
		Let $G$ be a GPD.  Then ${_pI}\cong \mathbb{F}_p$, where ${_pI}$ stands for the submodule of $I$ consists of all elements of order $p$. As a result, $I$ is isomorphic either to  $\mathbb{Q}_p/\Z_p$ or $\mathbb{Z}/p^s$ for some natural number $s$.
	\end{prop}
	First we need the following Lemma: 
	\begin{lem}\label{induced commutes with pontrygain}
		Let $A$ be a finite $U$ module in ${_p\operatorname{Mod}(U)}$, for an open subgroup $U\leq_o G$. Then there is a natural isomorphism $\operatorname{Ind}^G_U(A^*)\cong {\operatorname{Ind}^G_U(A)}^*$. 
	\end{lem}
	\begin{proof}
		First we construct a natural homomorphism $F:\operatorname{Ind}^G_U(A^*)\to  {\operatorname{Ind}^G_U(A)}^*$. For every $\varphi\in \operatorname{Ind}^G_U(A^*)$ and $\psi\in \operatorname{Ind}^G_U(A)$ define $F(\varphi)(\psi)=\sum \varphi(g_i)(\psi(g_i))$ where $\{g_i\}$ is a set of representatives of $U$ in $G$. One checks that it doesn't depend on the choice of such a set. Indeed, let $\{u_ig_i\}$ be another set of representatives. Then $$F(\varphi)(\psi)=\sum \varphi(u_ig_i)(\psi(u_ig_i))=\sum u_i\varphi(g_i)(u_i\psi(g_i))=\sum u_iu_i^{-1}\varphi(g_i)(\psi(g_i))=\sum\varphi(g_i)(\psi(g_i))$$ Now we show that this is indeed a $G$-map. Let $x\in G$. $$xF(\varphi)(\psi)=F(\varphi)(x^{-1}\psi)=\sum \varphi(g_i)(x^{-1}\psi(g_i))$$ $$=\sum \varphi(g_i)(\psi(x^{-1}g_i))=\sum \varphi(xg_i)(\psi(g_i))=\sum x\varphi(g_i)(\psi(g_i))=F(x\varphi)(\psi)$$ Next we show that $F$ is injective. Let $\varphi_1\ne \varphi_2\in \operatorname{Ind}^G_U(A^*)$. There exists some $g_i$ such that $\varphi_1(g_i)\ne \varphi_2(g_i)$. Hence there exists some $a\in A$ such that $\varphi_1(g_i)(a)\ne \varphi_2(g_i)(a)$. Construct $\psi\in \operatorname{Ind}^G_U(A) $ by $\psi(g_iU)=a$ and the zero function elsewhere. Then $F(\varphi_1)(\psi)\ne F(\varphi_2)(\psi)$. We left to show that $F$ is surjective. It is equivalent to show that $|\operatorname{Ind}^G_U(A^*)|=|{\operatorname{Ind}^G_U(A)}^*|$. We prove by induction on the size of $A$. For $A=\F_p$ this is immediate. Now assume that $|A|>p$ and look at the exact sequence $$0\to A_0\to A \to \F_p\to 0$$ Since $(-)^*$ and $\operatorname{Ind}^G_U(-)$ are both exact functors, we get two exact sequences $$0\to \operatorname{Ind}^G_U(\F_p)^*\to \operatorname{Ind}^G_U(A)^*\to   \operatorname{Ind}^G_U(A_0)^*\to 0$$ and $$0\to \operatorname{Ind}^G_U(\F_p^*)\to \operatorname{Ind}^G_U(A^*)\to   \operatorname{Ind}^G_U(A_0^*)\to 0$$ Hence $|\operatorname{Ind}^G_U(A^*)|=|\operatorname{Ind}^G_U(A_0^*)|\cdot |\operatorname{Ind}^G_U(\F_p^*)|$ and $|\operatorname{Ind}^G_U(A)^*|=|\operatorname{Ind}^G_U(A_0)^*|\cdot |\operatorname{Ind}^G_U(\F_p)^*|$ and by induction hypothesis we are done.
	\end{proof}
	\begin{proof}[Proof of Proposition \ref{dualizing module}]
		By \cite[Corollary 3.4.7]{neukirch2013cohomology}, ${_pI}\cong {\dirlim}_{U\leq o G} H^n(U,\F_p)^*$. Since $_pI$ cannot be trivial, it is enough to show that $H^n(U,\F_p)^*\cong \F_p$. Indeed, $\F_p\cong H^0(U,\F_p)\cong H^0(G,\operatorname{Ind}_U^G(\F_p))\cong H^n(G,(\operatorname{Ind}_U^G(\F_p))^*)^*\cong H^n(G,\operatorname{Ind}_U^G({\F_p}^*))^*\cong H^n(U,{F_p}^*)^*\cong H^n(U,\F_p)^*$. The third isomorphism follows from Proposition \ref{Isomorphism for 0} while the fourth one follows from Lemma \ref{induced commutes with pontrygain}.
	\end{proof}
	\begin{cor}
		A closed subgroup of infinite index of a GPD group of dimension  $n$ has cohomological dimension $<n$.
	\end{cor}
	\begin{proof}
		This is the same proof that appears in \cite{labute1966demuvskin} for countably ranked Demushkin groups. Let $H$ be a closed subgroup of infinite index in $G$. Then $H=\bigcap U_i$ the intersection of infinite strictly decreasing direct system. Hence $H^n(H,\F_p)=\dirlim H^n(U_i,\F_p)$. Let $U_j<U_i$. Then $\operatorname{cor}^{U_j}_{U_i}\circ\operatorname{res}^{U_i}_{U_j}:H^n(U_i,\F_p)\to H^n(U_i,\F_p)=[U_i:U_j]=0$. However, since by Corollary \ref{cohomological dimension} $cd(G)=n$, the same holds for every open subgroup of $G$, hence \cite[p. I-20, Lemma 4]{serre1979galois} implies that $\operatorname{cor}^{U_j}_{U_i}$ is surjective. By the proof of Proposition \ref{dualizing module}, $H^n(U_i,\F_p)\cong \F_p$ for every $i$, and hence the corestriction maps are also bijective. Thus $\operatorname{res}^{U_i}_{U_j}=0$.
	\end{proof}

	\section{Demushkin groups and extensions}

	For finitely generated Demushkin groups we have the following equivalence criteria:
	\begin{thm}(\cite[Theorem 3.7.2]{neukirch2013cohomology})
		Let $G$ be a finitely generated pro-$p$ group. The following conditions are equivalent:
		\begin{enumerate}
			\item $G$ is a Demushkin group.
			\item $\operatorname{cd}(G)=2$ and $I\cong \Q_p/\Z_p$.
			\item $\operatorname{cd}(G)=2$ and ${_pI}\cong \F_p$.
		\end{enumerate}
	\end{thm}
	For Demushkin groups of arbitrary rank we already know that the dualizing module may not be isomorphic to $\mathbb{Q}_p/\mathbb{Z}_p$. However, in (\cite{bar2024demushkin}) it has been shown that the two remaining conditions are still equivalent. I.e, we have the following theorem: 
	\begin{thm}(\cite[Theorem 31]{bar2024demushkin})
		Let $G$ be a pro-$p$ group of arbitrary rank. The following conditions are equivalent:
		\begin{enumerate}
			\item $G$ is a Demushkin group.
			\item $\operatorname{cd}(G)=2$ and ${_pI}\cong \F_p$.
		\end{enumerate}
	\end{thm}
	We use this criterion to prove the following theorem:
	\begin{thm}\label{GPD extension yeilds Demushkin}
		Let $1\to H\to G\to G/H\to 1$ be an exact sequence of pro-$p$ groups such that $G$ is a GPD group of dimension $n+2$. 
		\begin{enumerate}
			\item If $H$ is a GPD group of dimension $n$ and $\cd(G/H)<\infty$ then $G/H$ is a Demushkin group.
			\item If $G/H$ is a GPD group of dimension $n$ and $|H^m(H,\F_p)|<\infty$ for $m=\cd(H)$ then $H$ is a Demushkin group.
		\end{enumerate} 
	\end{thm}
	First we need the following observation:
	\begin{observation}\label{cohomology groups are finite}
		Let $H$ be a pro-$p$ group and $i$ an index such that $H^i(H,\F_p)$ is finite. Then for every finite $A\in{_p}\operatorname{Mod}(H)$, $H^i(H,A)$ is finite. Indeed, use induction on the size of $A$ and look at the exact sequence $0\to A_0\to A\to \F_p\to 0$. The claim follows from the induced exact sequence $H^i(H,A_0)\to H^i(H,A)\to H^i(H,\F_p)$.
	\end{observation}
	\begin{proof}[Proof of Theorem \ref{GPD extension yeilds Demushkin}]
		\begin{enumerate}
			
			\item First we show that $\cd(G/H)=2$. Since $\cd(G)\leq \cd(H)+\cd(G/H)$ holds for every exact sequence of pro-$p$ groups, we observe that $\cd(G/H)\geq 2$. Assume that $\cd(G/H)=m>2$. Look at the Hoschild-Serre spectral sequence $H^p(G/H,H^q(H,\F_p))\Rightarrow H^{p+q}(G,\F_p)$. Then $$E_2^{mn}\cong H^m(G/H,H^n(H,\F_p))\cong H^m(G/H,\F_p)\ne0.$$ The second isomorphism follows since $H$ is a GPD group of dimension $n$, and thus $H^n(H,\F_p)\cong \F_p$. By definition of cohomological dimension, $$E_2^{m+r,n-r+1}=E_2^{m-r,n+r-1}=0$$ so we get by induction that $E_r^{mn}=E_2^{mn}\ne0$ for all $r\geq 2$. Hence $H^{n+m}(G,\F_p)$ has a nontrivial graded piece- a contradiction.
			
			Now we show that $_pI\cong \F_p$ where $I$ denotes the dualizing module of $G/H$. Recall that by \cite[Corollary 3.4.7]{neukirch2013cohomology} $$_pI\cong {\dirlim}_{\operatorname{cor}^*}H^n(U/H,\F_p)^*$$ where $U$ runs over the set of all open subgroups of $G$ containing $H$. We shall prove that for every open subgroup $U$ of $G$ containing $H$, $H^n(U/H,\F_p)\cong \F_p$. That will imply that $_pI$ is a nontrivial subgroup of $\F_p$ and hence we are done. Let $U$ be an open subgroup of $G$ containing $H$. Look at the exact sequence $1\to H\to U\to U/H\to 1$. Recall that by the proof of Proposition \ref{dualizing module} every open subgroup $U$ of a GPD group of dimension $m$ satisfies $H^m(U,\F_p))\cong \F_p$. Moreover, $\cd(U/H)=\cd(G/H)=2$. Look at the Hoschild-Serre spectral sequence $$H^i(U/H,H^j(H,\F_p))\Rightarrow H^{i+j}(U,\F_p).$$ By cohomological dimensions, $$H^2(U/H,H^n(H,\F_p))\cong E_2^{2,n}\cong E_{\infty}^{2,n}\cong {\operatorname{gr}}_2H^{n+2}(U,\F_p)$$ is the only nontrivial piece of $H^{n+2}(U,\F_p)\cong \F_p$, so $ H^2(U/H,H^n(H,\F_p))\cong H^{n+2}(U,\F_p)\cong \F_p$. But $H^n(H,\F_p)\cong \F_p$ and we are done.
			
			\item First we show that $\cd(H)=2$. Since $\cd(H)\leq \cd(G)$ we observe that $\cd(H)$ is finite. First conclude that $\cd(H)\geq 2$ since $$\cd(G)\leq \cd(H)+\cd(G/H).$$ Now assume $\cd(H)=m>2$. Look at the Hoschild-Serre spectral sequence $H^i(G/H,H^j(H,\F_p))\Rightarrow H^{i+j}(G,\F_p)$. Then $$E_2^{nm}\cong H^n(G/H,H^m(H,\F_p))\ne0.$$ The intriviality follows from \cite[Corollary 7.1.7]{ribes2000profinite}, since $\cd(G/H)=n$ by Corollary \ref{cohomological dimension} and the assumption on $H^m(H,\F_p)$.   By definition of cohomological dimension, $$E_2^{n+r,m-r+1}=E_2^{n-r,m+r-1}=0$$ so we get by induction that $E_r^{nm}=E_2^{nm}\ne0$ for all $r\geq 2$. Hence $H^{n+m}(G,\F_p)$ has a nontrivial graded piece- a contradiction.
			
			We left to show that $_pI\cong \F_p$, where $I$ denotes the dializing module of $H$. Recall that ${_pI}\cong {\dirlim}_{V\leq _o H}H^2(V,\F_p)^*$ where $V$ runs over the set of open subgroups of $H$. Clearly, it is enough to look at the set of open subgroups of the form $H\cap U$ where $U$ is an open subgroup of $G$, since this set is cofinal in the set of all open subgroups of $H$. Let $U$ be an open subgroup of $G$. Then $\cd(U)=n+2$ and $H^{n+2}(U,\F_p)\cong \F_p$. Similarly,  $\cd(UH/H)=n$ and $H^{n}(UH/H,\F_p)\cong \F_p$. Moreover, $\cd(U\cap H)=\cd(H)=2$. Look at the Hoschild-Serre spectral sequence $$H^i(UH/H,H^j(U\cap H,\F_p))\Rightarrow H^{i+j}(U,\F_p).$$ By cohomological dimensions, $$H^n(UH/H,H^2(H\cap U,\F_p))\cong E_2^{n,2}\cong E_{\infty}^{n,2}\cong {\operatorname{gr}}_nH^{n+2}(U,\F_p)$$ is the only nontrivial piece of $H^{n+2}(G,\F_p)$, so $$ H^n(UH/H,H^2(H\cap U,\F_p))\cong H^{n+2}(U,\F_p)\cong \F_p.$$ 
			By Lemmas \ref{Isomorphism for 0} and \ref{induced commutes with pontrygain}, $$H^n(UH/H,H^2(H\cap U,\F_p))\cong H^0(UH/H,H^2(H\cap U,\F_p)^*)^*.$$ Indeed, it is enough to show that $H^2(H\cap U,\F_p)$ is a finite module. By taking the induced module from $H\cap U$ to $H$, it is enough to show that for every finite $A\in _p\operatorname{Mod}(H)$, $H^2(H,A)$ is finite. This follows from Observation \ref{cohomology groups are finite} due to the assumption that $H^2(H,\F_p)$ is finite. Thus,  $(H^2(H\cap U,\F_p)^*)^{UH/H}\cong \F_p$.
			
			Now, since $H^2(H\cap U,\F_p)^*$ is a finite module over $UH/H$, there is some $V\leq_o G$, $V\leq U$ such that $VH/H$ acts trivially on $H^2(H\cap U,\F_p)^*$. Look at the image of $H^2(H\cap U,\F_p)^*$ at $H^2(H\cap V,\F_p)^*$, we conclude that $\operatorname{cor}^*( H^2(H\cap U,\F_p)^*)\leq (H^2(H\cap V,\F_p)^*)^{VH/H}\cong \F_p$. Hence $_pI\cong \dirlim \F_p$.
			Since ${_pI}$ cannot be trivial, we are done.

		\end{enumerate}
	\end{proof}
	\begin{rem}
		In the above proof we used the following compatibility: let $g\in V$ and $F\in H^k(U\cap H,\F_p)^*$, then $g\operatorname{cor}^*(F)=\operatorname{cor}^*(gF)$. Indeed, let $f\in H^i(V\cap H,\F_p)$, then $$ (g\operatorname{cor}^*(F))(f)=(\operatorname{cor}^*(F))(g^{-1}f)=F(\operatorname{cor}(g^{-1}f))$$ while  $$\operatorname{cor}^*(gF)(f)=(gF)(\operatorname{cor}(f))=F(g^{-1}\operatorname{cor}(f)).$$ Hence it is enough to show that $$\operatorname{cor}(gf)=g\operatorname{cor}(f)$$ for every $g\in V$. We prove the compatibility in the level of homogeneous cochains. Let $\sigma_0,...,\sigma_k\in H\cap V$ and choose a set of representatives $c_1,...,c_n$ for $V\cap H$ in $U\cap H$. Recall that $H$ is normal, hence for every $g\in V$, ${c_1}^g,...,{c_n}^g$ is also a set of representatives for $V\cap H$ in $U\cap H$. Observe that $$({c_i\sigma_j\widetilde{c_i{\sigma_j}^{-1}}})^{g}={{c_i}^g{\sigma_j}^g\widetilde{{c_i}^g{\sigma_j}^g}  }$$ Thus $$\operatorname{cor}(gf)(\sigma_0,....,\sigma_k)=\sum_{c_i} gf(c_i\sigma_0\widetilde{c_i{\sigma_0}}^{-1},...,c_i\sigma_k\widetilde{c_i{\sigma_k}}^{-1})=$$ $$\sum_{c_i} f(({c_i}\sigma_0\widetilde{c_i{\sigma_0}}^{-1})^g,...,({c_i}\sigma_k\widetilde{c_i{\sigma_k}}^{-1})^g)=$$ $$\sum_{{c_i}^g} f(({c_i}^g{\sigma_0}^g\widetilde{{c_i}^g{\sigma_0}^{g}}^{-1}),...,{c_i}^g{\sigma_k}^g\widetilde{{c_i}^g{\sigma_k}^g}^{-1})=$$ $$(\operatorname{cor}(f))({\sigma_0}^g,...,{\sigma_k}^g)= g(\operatorname{cor}(f))(\sigma_0,...,\sigma_k)$$
	\end{rem}

	\bibliographystyle{plain}

\end{document}